\documentclass{article}

\usepackage[english]{babel}

\usepackage[letterpaper,top=2cm,bottom=2cm,left=3cm,right=3cm,marginparwidth=1.75cm]{geometry}

\usepackage{amsmath,amssymb,amsthm}
\usepackage{graphicx}
\usepackage[colorlinks=true, allcolors=blue]{hyperref}
\usepackage{setspace}
\usepackage{xcolor}%

\newtheorem{theorem}{Theorem}
\newtheorem{corollary}[theorem]{Corollary}
\newtheorem{lemma}{Lemma}

\newtheorem{definition}{Definition}

\title{From crank to congruences}
\author{Tewodros Amdeberhan and Mircea Merca}

\date{}

\begin{document}
\maketitle

\begin{abstract}  In this paper, we investigate the arithmetic properties of the difference between the number of partitions of a positive integer $n$ with even crank and those with odd crank, denoted $C(n) = c_e(n) - c_o(n)$. Inspired by Ramanujan’s classical congruences for the partition function $p(n)$, we establish a Ramanujan-type congruence 
 for $C(n)$, proving that $C(5n+4) \equiv 0 \pmod{5}$. Further, we study the generating function $\sum_{n=0}^\infty a(n)\, q^n = \frac{(-q; q)^2_\infty}{(q; q)_\infty}$, which arises naturally in this context, and provide multiple combinatorial interpretations for the sequence $a(n)$. We then offer a complete characterization of the values $a(n) \mod 2^m$ for $m = 1, 2, 3, 4$, highlighting their connection to generalized pentagonal numbers. Using computational methods and modular forms, we also derive new identities and congruences, including $a(7n+2) \equiv 0 \pmod{7}$, expanding the scope of partition congruences in arithmetic progressions. These results build upon classical techniques and recent computational advances, revealing deep combinatorial and modular structure within partition functions.
\end{abstract}

\section{Introduction}

A partition of a positive integer $n$ is any non-increasing sequence of positive
integers whose sum is $n$ \cite{Andrews98}. Let $p(n)$ denote the number of partitions of $n$ with the usual convention that
$p(0)=1$ and $p(n)=0$ when $n$ is not a non-negative integer. 
In 1919, Ramanujan \cite{Ramanujan} announced three elegant congruences satisfied by the partition function $p(n)$. These results reveal a remarkable arithmetic regularity, showing that for every non-negative integer $k$, the partition function $p(k)$ vanishes modulo $5$, $7$, and $11$ when $k$ is of the forms $5n + 4$, $7n + 5$, and $11n + 6$, respectively, i.e.,
\begin{align*}
p(5n+4) &\equiv 0 \pmod 5\\
p(7n+5) &\equiv 0 \pmod 7\\
p(11n+6) &\equiv 0 \pmod {11}.
\end{align*}

In order to explain the last two congruences combinatorially, Dyson \cite{Dyson}
introduced the rank of a partition. The rank of a partition is defined to be its
largest part minus the number of its parts.

In 1988, Andrews and Garvan \cite{Andrews88} defined the crank of an integer partition as follows. 
The crank of a partition is the largest part of the partition if there are no ones as parts and otherwise is the number of parts larger than the number of ones minus the number of ones. More precisely,
for a partition $\lambda=[\lambda_1,\lambda_2,\ldots,\lambda_k]$
let $\ell(\lambda)$ denote the largest part of $\lambda$, 
$\omega(\lambda)$ denote the number of $1$'s in $\lambda$, and $\mu(\lambda)$ denote the number of parts of $\lambda$ larger than $\omega(\lambda)$. The crank $c(\lambda)$ is given by
$$
c(\lambda) = \begin{cases}
\ell(\lambda), & \text{if $\omega(\lambda)=0,$}\\
\mu(\lambda)-\omega(\lambda),& \text{if $\omega(\lambda)>0.$}
\end{cases}
$$

	\begin{definition}
		Let $n$ be a non-negative integer. We define:
		\begin{enumerate}
			\item  $c_e(n)$ is the number of partitions of $n$ with even crank.
			\item  $c_o(n)$ is the number of partitions of $n$ with odd crank.
			\item  $C(n):=c_e(n)-c_o(n)$. 
		\end{enumerate}
	\end{definition}

\medskip
From \cite{Andrews18}, with a small corection, we have the identity:
$$
\sum_{n=0}^\infty C(n)\, q^n = 
2q+\frac{(q;q)_\infty}{(-q;q)_\infty^2},
$$
where 
the standard $q$-Pochhammer symbol $(a; q)_\infty$ is given by:
$$
(a; q)_\infty = \prod_{n=0}^\infty (1 - a q^n).
$$
We assume that $q$ is a complex number with $|q| < 1$.

As stated in \cite[Theorem 2.3.4]{Berndt}, Ramanujan's first congruence can be derived from the elegant identity:
$$
\sum_{n=0}^\infty p(5n + 4)\, q^n = 
5\cdot \frac{(q^5;q^5)^5_\infty}{(q;q)^6_\infty}.
$$
Inspired by this result, we observe the following analogous identity.

\medskip
\begin{theorem} \label{Th:1}
    For $|q|<1$, we have
        $$
    \sum_{n=0}^\infty C(5n+4) \, q^{5n+4} 
    =5q^4\cdot \frac{(q^5;q^5)^2_\infty\,(q^{25};q^{25})_\infty\,(q^{50};q^{50})^2_\infty}{(q^{10};q^{10})^4_\infty}.
    $$
\end{theorem}

\begin{lemma} \label{Lm:1} Let $\phi(q):= \prod_{n = 1}^{\infty}(1 - q^{n}) = (q; q)_{\infty}$. We have the quintisection expansions 
\begin{align*} (a) \quad \phi^2(q^2)
&=\frac{\phi^2(q^{50})\, (q^{20};q^{50})_{\infty}^2(q^{30};q^{50})_{\infty}^2}{(q^{10};q^{50})_{\infty}^2(q^{40};q^{50})_{\infty}^2}
+2q^6\, \frac{\phi^2(q^{50})\, (q^{10};q^{50})_{\infty}(q^{40};q^{50})_{\infty}}{(q^{20};q^{50})_{\infty} (q^{30};q^{50})_{\infty}}  \\
&- 2q^2\, \frac{\phi^2(q^{50})\, (q^{20};q^{50})_{\infty}(q^{30};q^{50})_{\infty}}{(q^{10};q^{50})_{\infty} (q^{40};q^{50})_{\infty}} 
+q^8\, \frac{\phi^2(q^{50})\, (q^{10};q^{50})_{\infty}^2(q^{40};q^{50})_{\infty}^2 }{(q^{20};q^{50})_{\infty}^2(q^{30};q^{50})_{\infty}^2} \\
& -q^4\, \phi^2(q^{50}), \\
(b) \ \, \quad \phi^3(q)
&=\frac{\phi^3(q^{25})\, (q^{15};q^{25})_{\infty}^3 (q^{10};q^{25})_{\infty}^3} {(q^{20};q^{25})_{\infty}^3(q^{5};q^{25})_{\infty}^3} 
-3q^5\, \frac{\phi^3(q^{25})\, (q^{20};q^{25})_{\infty}^2 (q^{5};q^{25})_{\infty}^2} {(q^{15};q^{25})_{\infty}^2 (q^{10};q^{25})_{\infty}^2} \\
& - 3q\, \frac{\phi^3(q^{25})\, (q^{15};q^{25})_{\infty}^2 (q^{10};q^{25})_{\infty}^2} { (q^{20};q^{25})_{\infty}^2 (q^{5};q^{25})_{\infty}^2} 
- q^6\, \frac{\phi^3(q^{25})\, (q^{20};q^{25})_{\infty}^3(q^{5};q^{25})_{\infty}^3} {(q^{15};q^{25})_{\infty}^3 (q^{10};q^{25})_{\infty}^3} \\
& +5q^3\, \phi^3(q^{25}).
\end{align*}
\end{lemma}

\begin{proof} We revive an identity due to Ramanujan (for instance, see Berndt's book \cite[pp. 81-82]{Berndt91}),
$$\phi(q)=\frac{\phi(q^{25})\, (q^{15};q^{25})_{\infty} (q^{10};q^{25})_{\infty}}{(q^{20};q^{25})_{\infty}(q^{5};q^{25})_{\infty}} -q\,  \phi(q^{25})
-q^2\, \frac{\phi(q^{25})\, (q^{5};q^{25})_{\infty} (q^{20};q^{25})_{\infty}}
{(q^{15};q^{25})_{\infty}(q^{10};q^{25})_{\infty}},$$
from which both (a) and (b) follow after grouping terms according to the powers of $q$ modulo $5$.
\end{proof}

\medskip

As a corollary of Theorem \ref{Th:1},  we derive the following Ramanujan type congruences modulo $5$.

\medskip
\begin{corollary} \label{Th:2}
    For any non-negative integer $n$, we have:
    \begin{enumerate}
        \item[(a)] 
        $C(5n+4) \equiv 0 \pmod 5$;
        \item[(b)] $c_e(5n+4) \equiv 0 \pmod 5$;
        \item[(c)] $c_o(5n+4) \equiv 0 \pmod 5$.
    \end{enumerate}
\end{corollary}
\medskip

In this context, we observe the following identity.

\medskip
\begin{theorem} \label{Th:3}
    For $|q|<1$, we have that
    $$
    \frac{(q^2;q^2)_\infty^3\, (q^{10};q^{10})_\infty}{(q;q)_\infty\, (q^5;q^5)_\infty^3}-q\,\frac{(q;q)_\infty\, (q^{10};q^{10})_\infty^5}{(q^2;q^2)_\infty\,(q^5;q^5)_\infty^5} = 1.
    $$
\end{theorem}

\medskip

In this paper, we will examine the arithmetic properties of the sequence $a(n)$, defined as the reciprocal of the infinite product arising from the generating function of $c_e(n) - c_o(n)$:
$$
\sum_{n=0}^\infty a(n)\,q^n = \frac{(-q;q)_\infty^2}{(q;q)_\infty}.
$$
The generous nature of this generating function allows us to remark multiple combinatorial interpretations for $a(n)$:
\begin{enumerate}
    \item [$\bullet$]  Considering Euler's identity
$$(-q;q)_\infty = \frac{1}{(q;q^2)_\infty},$$ we easily deduce that $a(n)$ is the number of partitions of $n$ in which each odd part is decorated using $3$ different colors. For example $a(3)=16$, because the partitions in question are:
   \begin{align*}
       & (3_3),\ (3_2),\ (3_1),\ (2,1_3),\ (2,1_2),\ (2,1_1),\ (1_3,1_3,1_3),\ (1_3,1_3,1_2),\  (1_3,1_3,1_1),\\
       & (1_3,1_2,1_2),\ (1_3,1_2,1_1),\ (1_3,1_1,1_1),\ (1_2,1_2,1_2),\ (1_2,1_2,1_1),\ (1_2,1_1,1_1),\ (1_1,1_1,1_1).
   \end{align*}
   \item [$\bullet$] Let $\nu_2(n)$ be the highest power of $2$ dividing $n$. Based on the formulas
   $$
   \prod_{n=1}\frac{1}{1-q^n} = \prod_{n=1}\prod_{m=0} (1+q^{n2^m}) = \prod_{n=1} (1+q^n)^{\nu_2(2n)},
   $$
   we deduce that
   $$
   \frac{(-q;q)_\infty^2}{(q;q)_\infty} = \prod_{n=1} (1+q^n)^{3+\nu_2(n)}.
   $$
   Thus $a(n)$ is the number of colored partitions of $n$ into distinct parts in which each part $k$ is decorated using $\nu_2(2^3k)$ different colors. For example $a(3)=16$, because the partitions in question are:
   \begin{align*}
       & (3_3),\ (3_2),\ (3_1),\ (2_4,1_3),\ (2_4,1_2),\ (2_4,1_1),\ (2_3,1_3),\ (2_3,1_2),\  (2_3,1_1),\\
       & (2_2,1_3),\ (2_2,1_2),\ (2_2,1_1),\ (2_1,1_3),\ (2_1,1_2),\ (2_1,1_1),\ (1_2,1_2,1_1).
   \end{align*}
   \item [$\bullet$] The product structure ${(q;q)^{-1}_\infty}\cdot (-q;q)^2_\infty$ can be understood as the product of two kinds of partitions:
   \begin{enumerate}
       \item [-] ${(q;q)^{-1}_\infty}$: This factor corresponds to choosing a standard partition.
       \item [-] $(-q;q)^2_\infty$: This factor introduces the coloring mechanism on partitions into distinct parts, where each part is decorated using $2$ different colors.  
   \end{enumerate}
   Thus $a(n)$ is the number of partitions of $n$ in which the first $2$ occurrences of their parts receive any (but distinct) of the $2$ colors. For example $a(3)=16$, because the partitions in question are:
   \begin{align*}
       & (3),\ (3_1),\ (3_2),\ (2,1),\ (2,1_1),\ (2,1_2),\ (2_1,1),\ (2_1,1_1),\  (2_1,1_2),\\
       & (2_2,1),\ (2_2,1_1),\ (2_2,1_2),\ (1,1,1),\ (1_1,1,1),\ (1_2,1,1),\ (1_2,1_1,1).
   \end{align*}
\end{enumerate}

We define the sequence $(\omega_k)_{k \geq 0}$ to be the generalized pentagonal numbers, given by the formula:
$$
\omega_k = \frac{1}{2} \left\lceil \frac{k}{2} \right\rceil \left\lceil \frac{3k+1}{2} \right\rceil,
$$
where $ \lceil x \rceil $ denotes the ceiling function, which rounds $ x $ up to the nearest integer. We remark that
$$\omega_{2n}=\frac{n(3n+1)}{2}\qquad\text{and}\qquad \omega_{2n-1}=\frac{n(3n-1)}{2}.$$
The following results provide a complete characterization of the congruences modulo $2^m$  of the sequence $a(n)$ for $m \in \{1, 2, 3, 4\}$.

\medskip
\begin{theorem} \label{Th:4}
  Let $n$ be a non-negative integer. 
  \begin{enumerate}
      \item[(a)] If $m\in\{1,2,3,4\}$, then $a(n) \equiv 0 \pmod {2^m}\ \iff\ n\not\in\{\omega_k \,|\, k\geq 0\}.  $
      \item [(b)]  If $m\in\{1,2,3,4\}$, then $a(\omega_n) \equiv 1 \pmod {2^m}\ \iff\ n\equiv \{-1,0\} \pmod {2^m}.$ 
      \item [(c)]  If $m\in\{2,3,4\}$, then $a(\omega_n) \equiv 3 \pmod {2^m}\ \iff\ n\equiv \{-2,1\} \pmod {2^m}.$ 
      \item [(d)]  If $m\in\{3,4\}$, then 
      \begin{enumerate}
          \item [($d_1$)] $a(\omega_n) \equiv 5 \pmod {2^m}\ \iff\ n\equiv \{-m-1,m\} \pmod {2^m};$
          \item [($d_2$)] $a(\omega_n) \equiv 7 \pmod {2^m} \ \iff \ n\equiv \{-3,2\} \pmod {2^m}.$
      \end{enumerate}
      \item [(e)] $a(\omega_n) \equiv 9 \pmod {2^4}\ \iff\ n\equiv \{-8,7\} \pmod {2^4}.$
      \item [(f)] $a(\omega_n) \equiv 11 \pmod {2^4}\ \iff\ n\equiv \{-7,6\} \pmod {2^4}.$
      \item [(g)] $a(\omega_n) \equiv 13 \pmod {2^4}\ \iff\ n\equiv \{-4,3\} \pmod {2^4}.$
      \item [(h)] $a(\omega_n) \equiv 15 \pmod {2^4}\ \iff\ n\equiv \{-6,5\} \pmod {2^4}.$
  \end{enumerate}
\end{theorem}

\medskip
\begin{corollary} \label{Th:5}
    Let $n$ be a non-negative integer. If $m\in\{1,2,3,4\}$, then $$a(w_{n+2^m}) \equiv a(\omega_n) \pmod {2^m}.$$    
\end{corollary}
\medskip

Using Mathematica package \texttt{RaduRK} developed by Nicolas Smoot \cite{Smoot}, we found the following generating function for $a(7n+2)$. 

\medskip
\begin{theorem}\label{Th:6}
    For $|q|<1$, we have:
    \begin{align*}
        \sum_{n=0}^\infty a(7n+2)\, q^n 
        = 7\left(\frac{1024\,f_2^8\,f_{14}^{18}}{f_1^{20}\,f_7^7}\,q^8 
        +\frac{1344\,f_2^9\,f_{14}^{11}}{f_1^{21}}\,q^6
        -\frac{1024\,f_2^{16}\,f_{14}^{10}}{f_1^{24}\, f_7^3}\,q^5
        +\frac{72\,f_2^{10}\,f_7^7\,f_{14}^4}{f_1^{22}}\,q^4 \right. \\
        \left. -\frac{320\,f_2^{17}\,f_7^4\,f_{14}^3}{f_1^{25}}\,q^3 
        -\frac{40\,f_2^{11}\,f_7^{14}}{f_1^{23}\,f_{14}^3}\,q^2
        +\frac{56\,f_2^{18}\, f_7^{11}}{f_1^{26}\, f_{14}^4}\,q
        +\frac{f_2^{12}\,f_7^{21}}{f_1^{24}\,f_{14}^{10}} \right),
    \end{align*}
    where 
$
f_a^b = (q^a; q^a)_\infty^b.
$
\end{theorem}
\medskip

As a corollary of this theorem, we derive the following Ramanujan type congruence modulo $7$.

\medskip
\begin{corollary} \label{Th:7}
        For any non-negative integer $n$, we have:
        $$a(7n+2) \equiv 0 \pmod 7.$$
\end{corollary}
\medskip

We remark the following identity.

\medskip
\begin{theorem}\label{Th:8}
    For $|q|<1$, it holds that
    \begin{align*}
        \frac{(q^2;q^2)_\infty^7\,(q^7;q^7)_\infty}{(q;q)_\infty^7\,(q^{14};q^{14})_\infty}
        -7q\,\frac{(q^7;q^7)_\infty^4}{(q;q)_\infty^4}
        +7q^3\,\frac{(q^{14};q^{14})_\infty^7}{(q;q)_\infty^3\,(q^2;q^2)_\infty\,(q^7;q^7)_\infty^3}=1.
    \end{align*}
\end{theorem}

\medskip

 In the following sections, we present rigorous proofs and supporting results for our theorems. 
	Section \ref{S2} establishes Theorem~\ref{Th:1} using classical $q$-series identities and quintisection expansions. 
	In Section \ref{S3}, we derive and analyze Theorem~\ref{Th:3}, a modular identity crucial for simplifying the generating function of the crank parity difference. 
	Section \ref{S4} explores the arithmetic properties and combinatorial significance of the sequence $a(n)$, which emerges naturally in this context. 
	Section \ref{S5} proves Theorem~\ref{Th:6} via an algorithmic approach grounded in the theory of modular functions, using the \texttt{RaduRK} package in Mathematica to implement Radu’s Ramanujan-Kolberg algorithm. 
	In Section \ref{S6}, we establish Theorem~\ref{Th:8} by examining the behavior of eta-quotients at cusps, leveraging classical modular form theory and Atkin-Lehner involutions. 
	Finally, Section \ref{S7} discusses several applications of Theorem~\ref{Th:1}.

\section{Proof of Theorem \ref{Th:1}} \label{S2}
We split the following relevant sum via $5$-section (i.e., based on powers of $q$ modulo $5$)
\begin{align*}
 \sum_{n = -\infty}^{\infty}(-1)^{n}q^{n(3n + 1)/2} =\phi(q) &=  e_{0} + e_{1} + e_{2} + e_{3} + e_{4}, \\
\phi^2(q^2) &= g_{0} + g_{1} + g_{2} + g_{3} + g_{4}, \\
 \sum_{n = 0}^{\infty}(-1)^{n}(2n + 1)q^{n(n + 1)/2} = \phi^{3}(q) &=   g_{0} + g_{1} + g_{2} + g_{3} + g_{4}, \\
 \sum_{n = 0}^{\infty}u(n)q^{n} = \frac{\phi^3(q)}{\phi^2(q^2)} &=  P_{0} + P_{1} + P_{2} + P_{3} + P_{4}.
\end{align*}
First of all note that by definition of $e_{s}$ we have $e_{3} = e_{4} = 0$ (because $n(3n + 1)/2$ is 
never equal to $3$ or $4$ modulo $5$). Similarly $n(n + 1)/2$ is never equal to $\pm 2$ modulo $5$ and hence $h_{2}= h_{4}= 0$. 

\smallskip
\noindent
Multiplying $\frac{\phi^3(q)}{\phi^2(q^2)}$ and $\phi^2(q^2)$, in the above equations, and combining terms according to powers of $q$ modulo $5$ we get the following set of equations
\begin{align}
g_{0}P_{0} + g_{4}P_{1} + g_{3}P_{2} + g_{2}P_{3} + g_{1}P_{4} &= h_0\notag\\
g_{1}P_{0} + g_{0}P_{1} + g_{4}P_{2} + g_{3}P_{3} + g_{2}P_{4} &= h_1\notag\\
g_{2}P_{0} + g_{1}P_{1} + g_{0}P_{2} + g_{4}P_{3} + g_{3}P_{4} &= 0\notag\\
g_{3}P_{0} + g_{2}P_{1} + g_{1}P_{2} + g_{0}P_{3} + g_{4}P_{4} &= h_3\notag\\
g_{4}P_{0} + g_{3}P_{1} + g_{2}P_{2} + g_{1}P_{3} + g_{0}P_{4} &= 0\notag
\end{align}
Our goal to is calculate $P_{4}$ which is given by $P_{4} = D_{4}/D$ where $D_{4}$ and $D$ are determinants given by
$$D = \begin{vmatrix}
g_{0} & g_{4} & g_{3} & g_{2} & g_{1}\\ 
g_{1} & g_{0} & g_{4} & g_{3} & g_{2}\\ 
g_{2} & g_{1} & g_{0} & g_{4} & g_{3}\\ 
g_{3} & g_{2} & g_{1} & g_{0} & g_{4}\\ 
g_{4} & g_{3} & g_{2} & g_{1} & g_{0}\end{vmatrix} \qquad \text{and} \qquad 
D_4 = \begin{vmatrix}g_{0} & g_{4} & g_{3} & g_{2} & h_0\\ 
g_{1} & g_{0} & g_{4} & g_{3} & h_1\\ 
g_{2} & g_{1} & g_{0} & g_{4} & 0 \\ 
g_{3} & g_{2} & g_{1} & g_{0} & h_3 \\ 
g_{4} & g_{3} & g_{2} & g_{1} & 0\end{vmatrix}.$$
The evaluation of determinant $D$ is aided by the fact that it is the determinant of a \emph{circulant matrix}, call it $A$. The determinant of a square matrix is the product of its eigenvalues and it is easy to find the eigenvalues of a circulant matrix. If $\omega$ is a $5^{\text{th}}$ root of unity (including $1$) then 
$$g_{0} + \omega g_{1} + \omega^{2}g_{2} + \omega^{3}g_{3} + \omega^{4}g_{4}$$ 
is an eigenvalue of $A$. Thus if $\omega$ is a primitive $5^{\text{th}}$ root of unity then 
$$\lambda_{t} = g_{0} + \omega^{t} g_{1} + \omega^{2t}g_{2} + \omega^{3t}g_{3} + \omega^{4t}g_{4}$$ 
gives all the eigenvalues of $A$ for $t = 0, 1, 2, 3, 4$. The determinant $D$ is therefore given by 
$$D = \prod_{t = 0}^{4}(g_{0} + \omega^{t} g_{1} + \omega^{2t}g_{2} + \omega^{3t}g_{3} + \omega^{4t}g_{4})
= \prod_{t = 0}^{4}\sum_{s = 0}^{4}\omega^{st}g_{s}$$ 
From the definition of $g_{s} = g_{s}(q)$ we can easily see that $\omega^{st}g_{s}(q) = g_{s}(\omega^{t}q)$ and hence
\begin{align}
D &= \prod_{t = 0}^{4}\sum_{s = 0}^{4}\omega^{st}g_{s} = \prod_{t = 0}^{4}\sum_{s = 0}^{4}g_{s}(\omega^{t}q) 
= \prod_{t = 0}^{4}\phi^2(\omega^{2t}q^2)\notag\\
&= \prod_{t = 0}^{4}\prod_{n = 1}^{\infty}(1 - \omega^{2nt}q^{2n})^2 
= \prod_{n = 1}^{\infty}\prod_{t = 0}^{4}(1 - \omega^{2tn}q^{2n})^2 \notag\\
&= \prod_{n \not\equiv 0\pmod{5}}(1 - q^{10n})^2 \prod_{n \equiv 0\pmod{5}}(1 - q^{2n})^{10}  \notag \\
&= \frac{\phi^{12}(q^{10})}{\phi(q^{50})^2}.
\end{align}
From the above calculations, we can see that 
$$\sum_{n = 0}^{\infty}u(5n + 4)q^{5n + 4} = P_{4} = \frac{D_{4}}{D} = D_4\cdot \frac{\phi(q^{50})^2}{\phi^{12}(q^{10})}.$$ 
The \emph{matrix determinant lemma} \cite{Ding07} states that if $A$ is a matrix, $v$ is a vector ($v^T$ its transpose) and $z$ is any indeterminate, then
$$\begin{vmatrix} v & A  \\ z&  v^{T}  \end{vmatrix}= v^T\text{adj}(A)\, v - z\, \vert A\vert.$$
Based on this fact, we compute
\begin{align*}
\delta(g_0,g_1,g_2,g_3,g_4): &= \begin{vmatrix}
g_1&g_0&g_4&g_3 \\
g_2&g_1&g_0&g_4 \\
g_3&g_2&g_1&g_0 \\
g_4&g_3&g_2&g_1
\end{vmatrix}
= - \begin{vmatrix}
{\color{red}g_1}&g_3&g_4&g_0 \\
{\color{red}g_2}&g_4&g_0&g_1 \\
{\color{red}g_3}&g_0&g_1&g_2 \\
g_4&{\color{red}g_1}&{\color{red}g_2}&{\color{red}g_3}
\end{vmatrix}  \\
&=- [{\color{red}g_1\, g_2\, g_3}] 
\begin{bmatrix}
g_0g_2-g_1^2 & g_1g_0-g_4g_2 & g_1g_4-g_0^2 \\
g_1g_0-g_4g_2 & g_3g_2-g_0^2 & g_0g_4-g_1g_3 \\
g_1g_4-g_0^2 & g_0g_4-g_1g_3 & g_0g_3-g_4^2
\end{bmatrix} \begin{bmatrix} {\color{red} g_1}  \\  {\color{red} g_2}  \\ {\color{red}  g_3 }\end{bmatrix} 
+ g_4\, \vert A\vert
\end{align*}
where $\vert A\vert = g_0g_2g_3+ 2g_0g_1g_4 -g_0^3 -g_1^2g_3-g_2g_4^2$.

\bigskip
\noindent
In view of this, the determinant $D_4$ can be given in the form
$$D_4
=h_0\cdot \delta(g_0,g_1,g_2,g_3,g_4)+h_1\cdot \delta(g_1,g_2,g_3,g_4,g_0)+h_3\cdot \delta(g_3,g_4,g_0,g_1,g_2).
$$
Finally, the evaluation of the determinant $D_{4}$ runs through the explicit expressions for $g_s$ and $h_s$ as given by Lemma~\ref{Lm:1}, in combination with routine simplifications,
which leads to the desired result.

\section{Proof of Theorem \ref{Th:3}} \label{S3}
Recall the \emph{Dedekind eta function} $\eta(q):=q^{\frac1{24}}\prod_{k\geq1}(1-q^k)$. We prove the equivalent form 
$$\frac{\eta(q^2)^3\cdot \eta(q^{10})}{\eta(q)\cdot \eta(q^5)^5}
 - \frac{\eta(q)\eta(q^{10})^5}{\eta(q^2)\cdot \eta(q^5)^5}=1$$
presented as an identity between eta-quotients, in $M_0(\Gamma_0(10))$, a weight $0$ level $10$ modular form. Both expressions on the left-hand side have a simple pole at the cusp 
$\frac15$ (under the image of the \emph{Atkin-Lehner involution} $W_2$ \cite{Atkin70}) and no other poles. This means there must be a linear combination of them which is constant, so just checking the constant and the vanishing of one other coefficient is enough.  

\smallskip
\noindent
Also, if you act by $W_2$, those eta-quotients become \emph{hauptmodln} for the congruence group $\Gamma_0(10)$ listed by Conway and Norton \cite{Conway79}, which must differ only by a constant. 

  \section{Proof of Theorem \ref{Th:4}} \label{S4}

To establish our theorem, we turn to two foundational results: Euler’s pentagonal number theorem 
 \begin{align*}
    (q;q)_\infty = \sum_{n=0}^\infty (-1)^{n(n+1)/2}\,q^{\omega_n} 
 \end{align*}
and a classical theta identity attributed to Gauss
  \begin{align}
    \frac{(q;q)_\infty}{(-q;q)_\infty} =1+2 \sum_{n=\infty}^\infty (-1)^n\,q^{n^2}. \label{eq:Gauss}
 \end{align}
These identities form the backbone of our approach to proving the theorem.

    \paragraph{Case $m=1$.}  Given that 
    $$
    \frac{(-q;q)_\infty}{(q;q)_\infty} = \left( 1+2 \sum_{n=\infty}^\infty (-1)^n\,q^{n^2} \right)^{-1}  \equiv 1 \pmod 2.
    $$
    we proceed as follows:
    \begin{align*}
        \sum_{n=0}^\infty a(n)\,q^n &= \frac{(-q;q)_\infty}{(q;q)_\infty}\cdot (-q;q)_\infty 
         \equiv (-q;q) \pmod 2
         \equiv  (q;q)_\infty \pmod 2 \\
        & \equiv \sum_{n=0}^\infty (-1)^{n(n+1)/2}\,q^{\omega_n} \pmod 2.
    \end{align*}

     \paragraph{Case $m=2$.} Expanding the inverse of \eqref{eq:Gauss} modulo $4$, we find
$$
\frac{(-q;q)_\infty}{(q;q)_\infty} \equiv 1 - 2 \sum_{n=1}^\infty (-1)^n\, q^{n^2} \equiv 1 + 2 \sum_{n=1}^\infty (-1)^n\, q^{n^2} \pmod{4}.
$$
Thus, we conclude that
$$
\frac{(-q;q)_\infty}{(q;q)_\infty} \equiv \frac{(q;q)_\infty}{(-q;q)_\infty} \pmod{4}.
$$
Now, we consider:
\begin{align*}
    \sum_{n=0}^\infty a(n)\,q^n &= \frac{(-q;q)_\infty}{(q;q)_\infty} \cdot (-q;q)_\infty \\
    &\equiv \frac{(q;q)_\infty \cdot (-q;q)_\infty}{(-q;q)_\infty} \pmod{4} \\
    &\equiv (q;q)_\infty \pmod{4}.
\end{align*}
This reasoning shows that
$$\sum_{n=0}^\infty a(n)\, q^n \equiv \sum_{n=0}^\infty (-1)^{n(n+1)/2}\,q^{\omega_n}  \pmod{4},$$ 
concluding the proof.

    \paragraph{Case $m=3$.} Expanding the inverse of \eqref{eq:Gauss} modulo $8$, we obtain
    \begin{align*}
        \frac{(-q;q)_\infty}{(q;q)_\infty} 
        & \equiv 1-2\sum_{n=1}^\infty (-1)^n\,q^{n^2} + \left( 2\sum_{n=1}^\infty (-1)^n\,q^{n^2} \right)^2  \pmod 8 \\
        & \equiv 1-2\sum_{n=1}^\infty (-1)^n\,q^{n^2} +4 \sum_{n=1}^\infty q^{2n^2} \pmod 8 \\
        & \equiv 1-2\sum_{n=1}^\infty (-1)^n\,q^{n^2} +4 \sum_{n=1}^\infty (-1)^n\,q^{2n^2} \pmod 8.
    \end{align*}
    This allows us to express
    \begin{align*}
         \frac{(-q;q)_\infty}{(q;q)_\infty} 
         & \equiv 2\,\frac{(q^2;q^2)_\infty}{(-q^2;q^2)_\infty} - \frac{(q;q)_\infty}{(-q;q)_\infty} \pmod 8.
    \end{align*}
    Then
    \begin{align*}
        \sum_{n=0}^\infty a(n)\,q^n 
        &= \frac{(-q;q)_\infty}{(q;q)_\infty}\cdot (-q;q)_\infty \\
        & \equiv 2\,\frac{(q^2;q^2)_\infty\,(-q;q)_\infty}{(-q^2;q^2)_\infty} - (q;q)_\infty \pmod 8\\
        & \equiv 2\,(q^2;q^2)_\infty\,(-q;q^2)_\infty - (q;q)_\infty \pmod 8\\
        & \equiv 2\,\frac{(q^2;q^2)_\infty}{(q;-q)_\infty} - (q;q)_\infty \pmod 8\\
        & \equiv 2\,(-q;-q)_\infty- (q;q)_\infty \pmod 8.
    \end{align*}
In other words, we have shown that
$$
\sum_{n=0}^\infty a(n)\, q^n \equiv \sum_{n=0}^\infty  (-1)^{n(n+1)/2}\,\big(2\,(-1)^{\omega_n} - 1 \big)\, q^{\omega_n} \pmod{8},
$$
thereby concluding the proof.
    
    \paragraph{Case $m=4$.} To summarize, we have:
\begin{align*}
\frac{(-q;q)_\infty^2}{(q;q)_\infty^2} 
&= \left( 1 + 2 \sum_{n=1}^\infty (-1)^n q^{n^2} \right)^{-2}
= \left(1 + 4 \sum_{n=1}^\infty (-1)^n q^{n^2} + \left( 2 \sum_{n=1}^\infty (-1)^n q^{n^2} \right)^2 \right)^{-1}.
\end{align*}
Modulo $16$, this simplifies to
    \begin{align*}
        \frac{(-q;q)_\infty^2}{(q;q)_\infty^2} 
        & \equiv 1-4\sum_{n=1}^\infty (-1)^n\,q^{n^2} - \left( 2\sum_{n=1}^\infty (-1)^n\,q^{n^2} \right)^2  \pmod {16} \\
        & \equiv 3-\frac{2\,(q;q)_\infty}{(-q;q)_\infty}-\left(\frac{(q;q)_\infty}{(-q;q)_\infty}-1 \right)^2 \pmod {16} \\
        & \equiv 2-\frac{(q;q)_\infty^2}{(-q;q)_\infty^2} \pmod {16}.
    \end{align*}
Next, we consider
$$
\sum_{n=0}^\infty a(n)\,q^n = \frac{(-q;q)_\infty^2}{(q;q)_\infty^2} \cdot (q;q)_\infty.
$$
Substituting, we find
$$
\sum_{n=0}^\infty a(n)\, q^n \equiv 2\, (q;q)_\infty - \frac{(q;q)_\infty^3}{(-q;q)_\infty^2} \pmod{16}.
$$
Considering \cite[eq. (32.6)]{Fine}), this yields
\begin{align*}
\sum_{n=0}^\infty a(n)\, q^n 
& \equiv 2 \sum_{n=-\infty}^\infty (-1)^n\, q^{n(3n-1)/2} + \sum_{n=-\infty}^\infty (6n-1)\, q^{n(3n-1)/2} \pmod{16}\\
& \equiv \sum_{n=-\infty}^\infty \left(6n - 1 + 2\,(-1)^n \right) q^{n(3n-1)/2}\pmod{16}.
\end{align*}
Thus, we conclude
$$\sum_{n=0}^\infty a(n)\, q^n \equiv \sum_{n=0}^\infty \bigg( 2\,(-1)^{n(n+1)/2} - (-1)^n\,(3n+1) + \frac{1-(-1)^n}{2} \bigg) \, q^{\omega_n} \pmod{16}.$$
This completes the proof.

\section{Proof of Theorem \ref{Th:6}} \label{S5}

Currently, this type of identity can be proven using computer algebra systems that implement Radu's Ramanujan-Kolberg algorithm \cite{Radu}. We make use of the Mathematica package \texttt{RaduRK}, developed by Nicolas Smoot \cite{Smoot}, which is known for its ease of use. The \texttt{RaduRK} package depends on \texttt{4ti2}, a software suite designed to address algebraic, geometric, and combinatorial problems involving linear spaces. To use the package, we follow the installation instructions outlined in \cite{Smoot} and activate it within a Mathematica session using the following command:
$$\text{\texttt{<<RaduRK`}}$$
It is essential to define the values of the two primary global variables, $q$ and $t$, before executing the program:
$$\text{\texttt{\{SetVar1[q],SetVar2[t]\}}}$$
The algorithmic verification of our identity is accomplished through the following procedure call:
$$\texttt{RK[14,2,\{-3,2\},7,2]}$$
Smoot's package presents the proof in the following format:
\allowdisplaybreaks{
	\begin{doublespace}
	\begin{align*}
	& \begin{array}{c|c}
	\texttt{N:} & 14 \\
	\hline
	\{\texttt{M,}(\texttt{r}_\delta)_{\delta|\texttt{M}}\}\texttt{:} & \{2,\{-3,2\}\} \\
	\hline
	\texttt{m:} & 7 \\
	\hline
	\texttt{P}_{\texttt{m,r}}(\texttt{j})\texttt{:} & \{2\} \\ 
	\hline \\ [-3.5ex]
	 \texttt{f}_\texttt{1}(\texttt{q})\texttt{:} & \dfrac{(q;q)_\infty^{20}\, (q^7;q^7)_\infty^7}{q^8\, (q^{2};q^{2})_\infty^8\, (q^{14};q^{14})_\infty^{18}} \\ [2ex]
	\hline \\ [-3.5ex]
	\texttt{t:} & \dfrac{(q^2;q^2)_\infty\, (q^7;q^7)_\infty^7}{q^2\, (q;q)_\infty\, (q^{14};q^{14})_\infty^7} \\ [2ex]
	\hline \\ [-3.5ex]
	 \texttt{AB:} & \big\{1,\dfrac{(q^2;q^2)_\infty^8\, (q^7;q^7)_\infty^4}{q^3\, (q;q)_\infty^4\, (q^{14};q^{14})_\infty^8} - \dfrac{4\,(q^2;q^2)_\infty\, (q^7;q^7)_\infty^7 }{q^2\, (q;q)_\infty\, (q^{14};q^{14})_\infty^7}\big\} \\ [2ex]
	\hline
	\{\texttt{p}_\texttt{g}(\texttt{t})\texttt{:g}\in \texttt{AB}\}\texttt{:} & \left\{7168 - 19264 t - 8456 t^2 + 1288 t^3 + 7 t^4,-7168 - 2240 t + 392 t^2\right\} \\
	\hline
	\texttt{Common Factor:} & 7 \\
	\end{array}
	\end{align*}
\end{doublespace}
}
As outlined in \cite{Andrews22}, the output can be interpreted as follows:
\begin{enumerate}
\item[$\bullet$] The first parameter in the procedure call $\texttt{RK[14,2,\{-3,2\},7,2]}$ sets $N = 14$, thereby defining the space of modular functions that the program will utilize:
$$
M(\Gamma_0(N)) := \text{the algebra of modular functions for } \Gamma_0(N).
$$
For detailed definitions of concepts like $\Gamma_0(N)$ and $M(\Gamma_0(N))$, as well as a thorough explanation of Radu's Ramanujan-Kolberg algorithm, please consult \cite{Paule}.
    \item[$\bullet$] The assignment $\{M, (r_\delta)_{\delta|M}\} = \{2, (-3, 2)\}$ is derived from the second and third entries of the procedure call $\texttt{RK[14,2,\{-3,2\},7,2]}$. This specifies $M = 2$ and $(r_\delta)_{\delta|2} = (r_1, r_2) = (-3, 2)$, such that
    $$
    \sum_{n=0}^\infty a(n)\, q^n = \prod_{\delta | \texttt M} (q^\delta;q^\delta)_\infty^{r_\delta} = \frac{(q^2;q^2)_\infty^2}{ (q;q)_\infty^3}.
    $$
    In the output expression $P_{m,r}(j)$ the abbreviation $r := (r_\delta)_{\delta|M}$ is used; i.e., here $r =(-3,2)$.
    \item[$\bullet$] The final two parameters in the procedure call $\texttt{RK[14,2,\{-3,2\},7,2]}$ correspond to the assignments $m = 7$ and $j = 2$, highlighting our emphasis on the generating function:
$$
\sum_{n=0}^\infty a(mn+j)\,q^n = \sum_{n=0}^\infty a(7n+2)\,q^n.
$$
The parameters $m$ and $j$ are utilized in the output expression $P_{m,r}(j)$; in this case, it is represented as $P_{7,r}(2)$, with $r = (-3, 2)$.
\item[$\bullet$]  The output $P_{m,r}(j) = P_{7,(-3,2)}(2) = \{2\}$ indicates the existence of an infinite product:
$$
f_1(q)=\frac{(q;q)_\infty^{20}\, (q^7;q^7)_\infty^7}{q^8\, (q^{2};q^{2})_\infty^8\, (q^{14};q^{14})_\infty^{18}}
$$
such that
$$
f_1(q)\sum_{n=0}^\infty a(7n+2)\,q^n \in M(\Gamma_0(N)),\quad\text{with}\quad N=14.
$$
\item[$\bullet$]  The output
\begin{align}
    & t=\frac{(q^2;q^2)_\infty\, (q^7;q^7)_\infty^7}{q^2\, (q;q)_\infty\, (q^{14};q^{14})_\infty^7},\notag \\ 
    & AB=\big\{1,\frac{(q^2;q^2)_\infty^8\, (q^7;q^7)_\infty^4}{q^3\, (q;q)_\infty^4\, (q^{14};q^{14})_\infty^8} - \dfrac{4\,(q^2;q^2)_\infty\, (q^7;q^7)_\infty^7 }{q^2\, (q;q)_\infty\, (q^{14};q^{14})_\infty^7}\big\},\notag \\
    & \{p_g(t):g\in AB\} = \left\{7168 - 19264 t - 8456 t^2 + 1288 t^3 + 7 t^4,-7168 - 2240 t + 392 t^2\right\}\label{eq5.2}
\end{align}
provides a solution to the following objective: 
find a modular function $t \in M(\Gamma_0(N))$ and
polynomials $p_g(t)$ such that
\begin{align}
    f_1(q)\sum_{n=0}^\infty a(7n+2)\,q^n  = \sum_{g\in AB} p_g(t)\cdot g.
\end{align}
Generally, the elements of the finite set $AB$ form a $\mathbb C[t]$-module basis of $M(\Gamma_0(N))$,
resp. of a large subspace of $M(\Gamma_0(N))$. 
The elements $g$ belonging to the set $AB$ are $\mathbb C$-linear combinations
of modular functions in $M(\Gamma_0(N))$ which are representable in infinite product form such
as $f_1(q)$ and $t$.
 In our case, the program delivers \eqref{eq5.2}, which
means
\begin{align*}
     & \frac{f_1^{20}\, f_7^7}{q^8 f_{2}^8\, f_{14}^{18}} \sum_{n=0}^\infty a(7n+2)\,q^n \\
     & \qquad = 7168 - 19264\, \frac{f_2\, f_7^7}{q^2\, f_1\, f_{14}^7} - 8456 \left( \frac{f_2\, f_7^7}{q^2\, f_1\, f_{14}^7} \right)^2 + 1288 \left( \frac{f_2\, f_7^7}{q^2\, f_1\, f_{14}^7} \right)^3 + 7 \left( \frac{f_2\, f_7^7}{q^2\, f_1\, f_{14}^7} \right)^4\\
     & \qquad\qquad\qquad\quad + \left( \frac{f_2^8\, f_7^4}{q^3\, f_1^4\, f_{14}^8} - \frac{4\,f_2\, f_7^7 }{q^2\, f_1\, f_{14}^7} \right) \left(-7168 - 2240\,\frac{f_2\, f_7^7}{q^2\, f_1\, f_{14}^7} + 392 \left( \frac{f_2\, f_7^7}{q^2\, f_1\, f_{14}^7} \right)^2 \right).
\end{align*}
This yields our identity on rearrangement.
\end{enumerate}

\section{Proof of Theorem \ref{Th:8}} \label{S6}
We prove the equivalent form 
$$\frac{\eta(q^2)^7\cdot \eta(q^7)}{\eta(q)^7\cdot \eta(q^{14})} - 7\, \frac{\eta(q^{7})^4}{\eta(q)^4}+7\, \frac{\eta(q^{14})^7}{\eta(q)^3\cdot\eta(q^2)\cdot\eta(q^7)^3}=1$$
in $M_0(\Gamma_0(14))$, which is a weight $0$ level $14$ modular form of an identity between eta-quotients. For background on such calculations see ~\cite[p. 18]{Ono04}.
Write the above identity as $f_1+f_2+f_3=1$.

\smallskip
\noindent
For each eta quotient, $f$, we associate a $4$-tuple $(a,b,c,d)$ giving the order of vanishing at each cusp, ordered by the Atkin-Lehner involutions $(W_1,W_2, W_7, W_{14})$ 
\cite{Atkin70}.
The $f_1$ gives $(0,0,2,-2)$, the $f_2$ gives  $(1,2,-1,-2)$, and $f_3$ gives $(3,0,-1,-2)$. Of course we can also associate a $4$-tuple to the constant function 1: $(0,0,0,0)$ Since these functions have no other poles, there is a non-trivial linear combination $g$ with orders at least $(1,0,0,-1)$. Here we could use a combination of $f_2$ and $f_3$ to get something with no pole under $W_7$, and then a multiple of $f_1$ to reduce the order of pole under $W_14$ to at least $-1$, and then subtract a constant to get vanishing at infinity. 
This function $g$ is either $0$, or $g\vert\,W_14 $is a hauptmodl (a meromorphic weight $0$ functions with a single pole at infinity). However, the modular curve $X_0(14)$ has genus $1$, so it cannot have a haputmodul. Thus $g$ is $0$. 

\smallskip
\noindent
This argument shows that there is a non-trivial relation between the functions. Obviously $1, f_2$ and $f_3$ are independent by $q$-expansion, and so $f_1$ can be found in terms of them by comparing the coefficients up to $q^3$.  We arrived at the conclusion.

\section{Applications of Theorem \ref{Th:1}} \label{S7}

\subsection{Euler's pentagonal number theorem}

This section is devoted to listing a number of applications to Theorem \ref{Th:1}. To minimize unduly replications, we only offer proofs to selected representatives of our results.

Denote $\Delta_k=\frac{k(k+1)}2$. Considering the theta series \cite[Eq. (0.44), p. 16]{Cooper}
\begin{align*}
(q;q)_\infty = \sum_{n=-\infty}^\infty (-1)^n\,q^{\omega_n},
\end{align*}
we derive the following corollaries.

\medskip
\begin{corollary} \label{Cor1} Let $n$ be a non-negative integer.
\begin{enumerate}
    \item [(a)] If $n\equiv 0 \pmod 2$, then $\sum\limits_{k\in\mathbb{Z}}C(5n+4-5\omega_k)\equiv 1 \pmod 2$  $\iff$  $n\in\{20\omega_j\vert j\geq0\}$.
    \item [(b)] If $n\equiv 0 \pmod 8$, then $\sum\limits_{k\in\mathbb{Z}}C(5n+4-25\omega_k)\equiv 1 \pmod 2$  $\iff$  $n\in\{8\Delta_j\vert j\geq0\}$.
    \item [(c)] If $n\equiv 4 \pmod 8$, then $\sum\limits_{k\in\mathbb{Z}}C(5n+4-25\omega_k)\equiv 1 \pmod 2$  $\iff$  $n\in\{40\Delta_j+4\vert j\geq0\}$.
\end{enumerate}
\end{corollary}
\medskip

\begin{corollary} \label{Cor2} Let $n$ be a non-negative integer.
\begin{enumerate}
    \item [(a)] If $n\equiv 1 \pmod 2$, then  $\sum\limits_{k\in\mathbb{Z}}C(5n+4-25\omega_k)\equiv 0 \pmod 2$.
    \item [(b)] If $n\equiv 6 \pmod 8$, then  $\sum\limits_{k\in\mathbb{Z}}C(5n+4-25\omega_k) \equiv 0 \pmod 2$.
    \item [(c)] If $n\equiv 5 \pmod 8$, then  $\frac15\sum\limits_{k\in\mathbb{Z}}(-1)^k C(5n+4-25\omega_k)\equiv 0 \pmod 5$.
    \item [(d)] If $n\not\equiv 0 \pmod 5$, then  $\frac15\sum\limits_{k\in\mathbb{Z}} (-1)^k C(50n+24-50\omega_k) \equiv 0 \pmod 5$.
    \item [(e)] If $n\not\equiv 2 \pmod 5$, then  $\frac15\sum\limits_{k\in\mathbb{Z}} (-1)^k C(50n+49-50\omega_k) \equiv 0 \pmod 5$.
\end{enumerate}
\end{corollary}
\medskip

\begin{proof}
Corollary \ref{Cor2} (a). We consider the sequence $A(n)$ defined by
$$
\sum_{n=0}^\infty A(n)\,q^n = \frac{f_1^2\,f_5^2\,f_{10}^2}{f_2^4}.
$$ 
It is clear that
\begin{align}
A(n)=\sum_{k=-\infty}^\infty \frac{(-1)^k}{5}\,C\left(5(n-5\omega_k)+4\right).
\end{align}
We need to show that $A(2n+1)\equiv 0 \pmod 2.$
Using the Mathematica package \texttt{RaduRK} with
$$\texttt{RK[20,10,\{2,-4,2,2\},2,1]},$$
we derive the following identity:
\begin{align*}
\sum_{n=0}^\infty A(2n+1)\,q^n
=-2\,\frac{f_2\,f_5^8\,f_{20}}{f_1^4\,f_4\,f_{10}^3}
+2\,q\,\frac{f_4^2\,f_5^3\,f_{20}^2}{f_1^3\,f_2\,f_{10}}
+2\,q^3\,\frac{f_2\,f_5^3\,f_{20}^6}{f_1^3\,f_4^2\,f_{10}^3}.
\end{align*}
The claim follows.
\end{proof}

\medskip

\begin{corollary} \label{Cor3} Let $n$ be a non-negative integer.
\begin{enumerate}
    \item[(a)] $\sum\limits_{k\in\mathbb{Z}}a(2n-\omega_k)\equiv 1 \pmod 2$   $\iff$ $n\in\{\omega_j\vert j\geq0\}$.
    \item[(b)] $\sum\limits_{k\in\mathbb{Z}}a(2n+1-2\omega_k)\equiv 0 \pmod 2$   $\iff$  $n\in\{\Delta_j\vert j\geq0\}$.
    \item[(c)] $\sum\limits_{k\in\mathbb{Z}}a(2n-5\omega_k)\equiv 1 \pmod 2$   $\iff$   $n\in\{\Delta_j\vert j\geq0\}$.
    \item[(d)] $\sum\limits_{k\in\mathbb{Z}}a(2n+1-5\omega_k)\equiv 1 \pmod 2$    $\iff$   $n\in\{5\Delta_j\vert j\geq0\}$.
\end{enumerate}
\end{corollary}

\medskip

\begin{corollary} \label{Cor4} Let $n$ be a non-negative integer.
\begin{enumerate}
    \item [(a)] $\sum\limits_{k\geq0} (-1)^k a(2n+1-\omega_k)\equiv 0 \pmod 2$.
    \item [(b)] $\sum\limits_{k\in\mathbb{Z}}a(3n+1-2\omega_k) \equiv 0 \pmod 3$.
    \item [(c)] $\sum\limits_{k\in\mathbb{Z}} a(3n+2-2\omega_k)\equiv 0 \pmod 6$.
    \item [(d)] $\sum\limits_{k\in\mathbb{Z}} a(2n+1-3\omega_k) \equiv 0 \pmod 3$.
\end{enumerate}
\end{corollary}

\medskip

\begin{proof}

	Corollary \ref{Cor3} (a) and Corollary \ref{Cor4} (a). We consider the sequence $A(n)$ defined by
	$$
	\sum_{n=0}^\infty A(n)\,q^n = \frac{f_2^2}{f_1^2}.
	$$
		It is clear that
	$$
	A(n)=\sum_{k=-\infty}^\infty (-1)^k\,a\left(n-\omega_k\right).
	$$
	The proof follows if we consider that
	\begin{align*}
	\frac{f_2^2}{f_1^2} = \frac{f_8^5}{f_2^3\,f_{16}^2}-2\,q\, \frac{f_4^2\,f_{16}^2}{f_2^3\,f_8} \qquad \text{and} \qquad
	\frac{f_4^5}{f_1^3\,f_{8}^2} \equiv f_1 \pmod 2.
	\end{align*}
	Corollary \ref{Cor3} (b) and Corollary \ref{Cor4} (b)-(c).  We consider the sequence $A(n)$ defined by
	$$
	\sum_{n=0}^\infty A(n)\,q^n = \frac{f_2^3}{f_1^3}.
	$$
	It is clear that
	$$
	A(n)=\sum_{k=-\infty}^\infty (-1)^k\,a\left(n-2\omega_k\right).
	$$
	The proof follows if we consider that:
	\begin{align*}
	 \frac{f_2^3}{f_1^3} \equiv \frac{f_2^2}{f_1} \pmod 2, \qquad
	 \sum_{n=0}^\infty A(3n+1)\,q^n =3\, \frac{f_2^4\,f_3^5}{f_1^8\,f_6}, \qquad
	 \sum_{n=0}^\infty A(3n+2)\,q^n =6\, \frac{f_2^3\,f_3^2\,f_6^2}{f_1^7}.
	\end{align*}
	Corollary \ref{Cor4} (d).  We consider the sequence $A(n)$ defined by
	$$
	\sum_{n=0}^\infty A(n)\,q^n = \frac{f_2^2\,f_3}{f_1^3}.
	$$
	It is evident now that
	$$
	A(n)=\sum_{k=-\infty}^\infty (-1)^k\,a\left(n-3\omega_k\right).
	$$
	The proof follows if we consider that
	\begin{align*}
	\frac{f_2^2\,f_3}{f_1^3} = \frac{f_4^6\,f_6^3}{f_2^7\,f_{12}^2} + 3\,q\,\frac{f_4^2\,f_6\,f_{12}^2}{f_2^5}.
	\end{align*}
	Corollary \ref{Cor3} (c)-(d). We consider the sequence $A(n)$ defined by
	$$
	\sum_{n=0}^\infty A(n)\,q^n = \frac{f_2^2\,f_5}{f_1^3}.
	$$
	Therefore, 
	$$
	A(n)=\sum_{k=-\infty}^\infty (-1)^k\,a\left(n-5\omega_k\right).
	$$
	The proof follows if we consider that
	\begin{align*}
	\frac{f_2^2\,f_5}{f_1^3} \equiv \frac{f_4^2}{f_2} + q\,\frac{f_{20}^2}{f_{10}} \pmod 2.
	\end{align*}
We conclude the argument.
\end{proof}

\subsection{Jacobi's identity}

Considering the Jacobi identity \cite[Eq. (0.46), p. 17]{Cooper}
\begin{align*}
(q;q)_\infty^3 = \sum_{n=0}^\infty (-1)^n\,(2n+1)\,q^{\Delta_n},
\end{align*}
we derive the following identity.

\medskip

\begin{corollary} \label{Cor5} Let $n$ be a non-negative integer.
\begin{enumerate}
    \item [(a)] If $n\equiv \{1,3\} \pmod 5$, then $\frac15\sum\limits_{k\geq0} (-1)^k\, (2k+1)\, C(5n+4-5\Delta_k) \equiv 0 \pmod 5$.
    \item [(b)] If $n\equiv \{2,3\} \pmod 5$, then $\sum\limits_{k\geq0} (-1)^k\, (2k+1)\, C(5n+4-10\Delta_k) =0$.
    \item [(c)] If $n\not\equiv 0 \pmod 5$, then $\sum\limits_{k\geq0} C(5n+4-10\Delta_k)\equiv 0 \pmod 2$.
    \item [(d)] If $n\equiv 1 \pmod 2$, then $\sum\limits_{k\geq0} C(5n+4-25\Delta_k) \equiv 0 \pmod 2$.
\end{enumerate}
\end{corollary}
\medskip

\begin{proof}
	(b)-(c). 	We consider the sequence $A(n)$ defined by
	$$
	\sum_{n=0}^\infty A(n)\,q^n = \frac{f_1^2\,f_5\,f_{10}^2}{f_2}
	$$
	It is clear that
	$$
	A(n)=\frac{1}{5} \sum_{k=0}^\infty (-1)^k\,(2k+1)\, C\left(5n+4-5k(k+1)\right).
	$$
	We need to show that
	$
	A(5n \pm 2)  = 0,
	$
	and
	$
	A(5n \pm 1) \equiv 0 \pmod 2.
	$
	The proof follows easily if we consider
	\begin{align*}
	\frac{f_1^2}{f_2} = \frac{f_{25}^2}{f_{50}}-2\,q\,(q^{15},q^{35},q^{50};q^{50})_\infty -2\,q^4\, (q^5,q^{45},q^{50};q^{50})_\infty.
	\end{align*}
		(d). 	We consider the sequence $A(n)$ defined by
	$$
	\sum_{n=0}^\infty A(n)\,q^n = \frac{f_1^2\,f_5^4\,f_{10}^2}{f_2^4}
	$$
	It is clear that
	$$
	A(n)=\frac{1}{5} \sum_{k=0}^\infty (-1)^k\,(2k+1)\, C\left(5n+4-\frac{25k(k+1)}{2}\right).
	$$
	We need to show that
	$
	A(2n+1) \equiv 0 \pmod 2.
	$
	On the other hand, we have
	$$
	A(2n+1) \equiv B(2n+1) \pmod 2,
	$$
	where the sequence $B(n)$ defined by
	$$
	\sum_{n=0}^\infty B(n)\,q^n = f_1^2\,f_5^4.
	$$
	The proof follows easily if we consider the identities:
	\begin{align*}
	f_1^2 = \frac{f_2\,f_8^5}{f_4^2\,f_{16}^2}-2\,q\,\frac{f_2\,f_{16}^2}{f_8} \qquad \text{and} \qquad
	f_1^4 = \frac{f_4^{10}}{f_2^2\,f_{8}^4}-4\,q\,\frac{f_2^2\,f_{8}^4}{f_4^2}.
	\end{align*}
We conclude the argument.
\end{proof}

\medskip
\begin{corollary} \label{Cor6} Let $n$ be a non-negative integer.
\begin{enumerate}
    \item [(a)] $\sum\limits_{k\geq0} (-1)^k\, (2k+1)\, a(2n+1-\Delta_k) = 0$.
    \item [(b)] $\sum\limits_{k\geq0}a(4n-\Delta_k) \equiv 1 \pmod 2$   $\iff$  $n\in\{\omega_j\vert j\geq0\}$.
    \item [(c)] $\sum\limits_{k\geq0} a(4n+2-\Delta_k)\equiv 0 \pmod 2$.
    \item [(d)] If $n\neq 1\pmod5$  then  $\sum\limits_{k\geq0} a(5n+2-2\Delta_k) \equiv 0 \pmod 2$.
\end{enumerate}
\end{corollary}
\medskip

\begin{proof}

	(a)-(c). We consider the sequence $A(n)$ defined by
	$$
	\sum_{n=0}^\infty A(n)\,q^n = f_2^2
	$$
	It is clear that
	$$
	A(n)=\sum_{k=0}^\infty (-1)^k\,(2k+1)\, a\left(n-\Delta_k\right).
	$$
	The proof follows considering that
	\begin{align*}
	f_2^2 = \frac{f_4\,f_{16}^5}{f_8^2\,f_{32}^2} - 2\,q^2\,\frac{f_4\,f_{32}^2}{f_{16}} \qquad \text{and} \qquad
   \frac{f_1\,f_{4}^5}{f_2^2\,f_{8}^2} \equiv f_1 \pmod 2.
	\end{align*}
	We conclude the argument.
\end{proof}

\subsection{Gauss theta series}

Considering the theta series identity \cite[Eq. (0.41), p. 16]{Cooper}
\begin{align*}
\frac{(q;q)_\infty^2}{(q^2;q^2)_\infty} =  \sum_{n=-\infty}^\infty (-1)^n\,q^{n^2},
\end{align*}
we derive the following corollary.

\medskip
\begin{corollary} 
	Let $n$ be a non-negative integer.
	\begin{enumerate}
		\item [(a)]	If $n\equiv 4 \pmod 5$, then
		$
		\frac{1}{5}\,\sum\limits_{k=-\infty}^\infty (-1)^k\,C\big(5n+4-5k^2\big)  \equiv 0 \pmod 5.
		$
				\item [(b)]	If $n\equiv 4 \pmod 5$, then
		$
		\frac{1}{5}\,\sum\limits_{k=-\infty}^\infty (-1)^k\,C\big(5n+4-10k^2\big)  \equiv 0 \pmod 5.
		$
	\end{enumerate}
\end{corollary}
\medskip

\noindent
Considering the theta series identity \cite[Eq. (0.45), p. 16]{Cooper}
\begin{align*}
\frac{(q^2;q^2)_\infty^2}{(q;q)_\infty} =  \sum_{n=0}^\infty q^{\Delta_n},
\end{align*}
we derive the following corollary.

\medskip
\begin{corollary} 
	Let $n$ be a non-negative integer.
	\begin{enumerate}
		\item [(a)]	If $n\equiv \{6,8\} \pmod{10}$, then 	
		$
		\sum\limits_{k=0}^\infty C\Big(5n+4-5\Delta_k\Big)  \equiv 0 \pmod 2.
		$
		\item [(b)]	If $n\equiv 1 \pmod{2}$, then 	
		$
		\sum\limits_{k=0}^\infty C\Big(5n+4-25\Delta_k\Big)  \equiv 0 \pmod 2.
		$
	\end{enumerate}
\end{corollary}
\medskip

\begin{proof}
	(b). We consider the sequence $A(n)$ defined by
	$$
	\sum_{n=0}^\infty A(n)\,q^n = \frac{f_1^2 f_{10}^4}{f_2^4}
	$$
	It is clear that
	$$
	A(n)=\frac{1}{5} \sum_{k=0}^\infty C\left(5\Big(n-5\Delta_k\Big)+4\right).
	$$
	We need to show that
	$A(2n+1)\equiv 0 \pmod 2$ whose proof follows easily if we consider that
	\begin{align*}
	\frac{f_1^2}{f_2} 
	= \frac{f_8^5}{f_4^2\,f_{16}^2}
	-2\,q\,\frac{f_{16}^2}{f_8}.
	\end{align*}
We conclude the argument.
\end{proof}

\subsection{Ramanujan theta functions}

Considering the theta identity \cite[Eq. (0.47), p. 17, with $q$ replaced by $-q$]{Cooper}
\begin{align*}
\frac{(q^2;q^2)_\infty^5}{(q;q)_\infty^2}
= \sum_{n=-\infty}^{\infty} (-1)^n\,(3n+1)\,q^{3n^2+2n}, 
\end{align*}
we derive the following corollary.

\medskip
\begin{corollary} 
If $n\equiv \{1,3\} \pmod {5}$, then	$\sum\limits_{k=-\infty}^\infty (-1)^k\,(3k+1)\, C\left(5n+4-5k(3k+2)\right)
		= 0.
		$
\end{corollary}
\medskip

\begin{proof}
		We consider the sequence $A(n)$ defined by
	$$
	\sum_{n=0}^\infty A(n)\,q^n = f_2\, f_5\,f_{10}^2
	$$
		It is clear that
	$$
	A(n)=\frac{1}{5} \sum_{k=-\infty}^\infty (-1)^k\,(3k+1)\, C\left(5\big(n-k(3k+2)\big)+4\right).
	$$
	The proof follows if we consider that
	\begin{align*}
	f_2\,f_{10} & = (q^{20},q^{30},q^{50};q^{50})_\infty^2 - f_{10}\,f_{50}\cdot 
         q^2 - (q^{10},q^{40},q^{50};q^{50})_\infty^2\cdot q^4 \qquad \text{or}   \\
	f_1\,f_{5} &= (q^{10},q^{15},q^{25};q^{25})_\infty^2 - f_{5}\,f_{25}\cdot 
           q - (q^{5},q^{20},q^{25};q^{25})_\infty^2\cdot q^2. 
	\end{align*}
The proof is complete.
\end{proof}

Considering the theta identity  \cite[Eq. (0.48), p. 17]{Cooper}
\begin{align*}
\frac{(q;q)_\infty^5}{(q^2;q^2)_\infty^2}
= \sum_{n=-\infty}^{\infty} (1-6n)\,q^{\omega_n}, 
\end{align*}
we derive the following corollary.

\medskip
\begin{corollary} 
	Let $n$ be a non-negative integer.
	\begin{enumerate}
		\item[(a)] 	If $n\equiv \{2,3\} \pmod {5}$, then	$\frac{1}{5}\,\sum\limits_{k=-\infty}^\infty (1-k)\, C\left(5n+4-5\omega_k\right)
\equiv 0 \pmod 5.$
		\item[(b)] If $n\equiv 9 \pmod {10}$, then		$\frac{1}{5}\, \sum\limits_{k=-\infty}^\infty (1-k)\, C\left(5n+4-25\omega_k\right)
\equiv 0 \pmod 5.
$
	\end{enumerate}
\end{corollary}
\medskip

\begin{proof}
	(b). We consider the sequence $A(n)$ defined by
	$$
	\sum_{n=0}^\infty A(n)\,q^n = \frac{f_1^2 f_{5}^6}{f_2^4}
	$$
	It is clear that
	$$
	A(n)=\frac{1}{5} \sum_{k=-\infty}^\infty (1-6k)\, C\left(5\Big(n-5\omega_k\Big)+4\right).
	$$
	We need to show that
	$A(10n+9)\equiv 0 \pmod {10}.$
	Using the Mathematica package \texttt{RaduRK} with
	$$\texttt{RK[20,10,\{2,-4,6,0\},10,9]}$$
	allows to derive the following identity:
	\begin{align*}
	\sum_{n=0}^\infty A(10n+9)\,q^n
	&=-10\,\frac{f_5^2}{f_1^{10}\,f_2^4} \left( \frac{17\,f_4^{38}\,f_{10}^4}{f_2^{20}\,f_{20}^6} 
	+\frac{617\,f_4^{31}\,f_5^5}{f_1\,f_2^{16}\,f_{20}^3}\,q 
	+\frac{6448\,f_4^{34}\,f_{10}^2}{\,f_2^{18}\,f_{20}^2}\,q^2
	+\frac{37948\,f_4^{27}\,f_5^5\,f_{20}}{f_1\,f_2^{14}\,f_{10}^2}\,q^3
	\right.\\
	& 
	+\frac{57143\,f_4^{30}\,f_{20}^2}
	{f_2^{16}}\,q^4
	+\frac{110960\,f_4^{23}\,f_5^5\,f_{20}^5}
	{f_1\,f_2^{12}\,f_{10}^4}\,q^5
	-\frac{331248\,f_4^{26}\,f_{20}^6}
	{f_2^{14}\,f_{10}^2}\,q^6 \\
    & 	-\frac{346100\,f_4^{19}\,f_5^5\,f_{20}^9}
    {f_1\,f_2^{10}\,f_{10}^6}\,q^7
    +\frac{422490\,f_4^{22}\,f_{20}^{10}}
    {f_2^{12}\,f_{10}^4}\,q^8
    +\frac{453450\,f_4^{15}\,f_5^5\,f_{20}^{13}}
    {f_1\,f_2^{8}\,f_{10}^8}\,q^9 \\
    &
    +\frac{471600\,f_4^{18}\,f_{20}^{14}}
    {f_2^{10}\,f_{10}^6}\,q^{10}
    -\frac{367500\,f_4^{11}\,f_5^5\,f_{20}^{17}}
    {f_1\,f_2^{6}\,f_{10}^{10}}\,q^{11}
	-\frac{1736450\,f_4^{14}\,f_{20}^{18}}
	{f_2^{8}\,f_{10}^8}\,q^{12} \\
	&
	-\frac{5000\,f_4^{7}\,f_5^5\,f_{20}^{21}}
	{f_1\,f_2^{4}\,f_{10}^{12}}\,q^{13}
	+\frac{1630000\,f_4^{10}\,f_{20}^{22}}
	{f_2^{6}\,f_{10}^{10}}\,q^{14}
	+\frac{162500\,f_4^{3}\,f_5^5\,f_{20}^{25}}
	{f_1\,f_2^{42}\,f_{10}^{14}}\,q^{15}\\
	&\left.
	-\frac{466875\,f_4^{6}\,f_{20}^{26}}
	{f_2^{4}\,f_{10}^{12}}\,q^{16}
	-\frac{46875\,f_5^5\,f_{20}^{29}}
	{f_1\,f_4\,f_{10}^{16}}\,q^{17}
	-\frac{100000\,f_4^{2}\,f_{20}^{30}}
	{f_2^{2}\,f_{10}^{14}}\,q^{18}
	+\frac{46875\,f_{20}^{34}}
	{f_4^{2}\,f_{10}^{16}}\,q^{20}
		\right)
	\end{align*}
	This concludes the proof.
\end{proof}

\bibliographystyle{alpha}
\bibliography{sample}

\begin{thebibliography}{Ram19}

\bibitem[AG88]{Andrews88}
George~E. Andrews and F.~G. Garvan.
\newblock Dyson's crank of a partition.
\newblock {\em Bull. Amer. Math. Soc. (N.S.)}, 18(2):167--171, 1988.

\bibitem[AL70]{Atkin70}
A.~O. Atkin and J.~Lehner.
\newblock {H}ecke operators on $\gamma_0(m)$.
\newblock {\em Mathematische Annalen}, 185(2):134--160, 1970.

\bibitem[And98]{Andrews98}
George~E. Andrews.
\newblock {\em The theory of partitions}.
\newblock Cambridge Mathematical Library. Cambridge University Press, Cambridge, 1998.
\newblock Reprint of the 1976 original.

\bibitem[And18]{Andrews18}
George~E. Andrews.
\newblock Integer partitions with even parts below odd parts and the mock theta functions.
\newblock {\em Ann. Comb.}, 22(3):433--445, 2018.

\bibitem[AP22]{Andrews22}
George~E. Andrews and Peter Paule.
\newblock Mac{M}ahon's partition analysis {XIII}: {S}chmidt type partitions and modular forms.
\newblock {\em J. Number Theory}, 234:95--119, 2022.

\bibitem[Ber91]{Berndt91}
Bruce~C. Berndt.
\newblock {\em {R}amanujan’s Notebooks, Part III}.
\newblock Springer-Verlag, New York, 1991.

\bibitem[Ber06]{Berndt}
Bruce~C. Berndt.
\newblock {\em Number theory in the spirit of {R}amanujan}, volume~34 of {\em Student Mathematical Library}.
\newblock American Mathematical Society, Providence, RI, 2006.

\bibitem[CN79]{Conway79}
J.~H. Conway and S.~P. Norton.
\newblock Monstrous moonshine.
\newblock {\em Bull. London Math. Soc.}, (3):308--339, 1979.

\bibitem[Coo17]{Cooper}
Shaun Cooper.
\newblock {\em Ramanujan's theta functions}.
\newblock Springer, Cham, 2017.

\bibitem[Dys44]{Dyson}
F.~J. Dyson.
\newblock Some guesses in the theory of partitions.
\newblock {\em Eureka}, 11(8):10--15, 1944.

\bibitem[DZ07]{Ding07}
J.~Ding and A.~Zhou.
\newblock Eigenvalues of rank–one updated matrices with some applications.
\newblock {\em Applied Math. Letters}, 20(16):1223--1226, 2007.

\bibitem[Fin88]{Fine}
Nathan~J. Fine.
\newblock {\em Basic hypergeometric series and applications}, volume~27 of {\em Mathematical Surveys and Monographs}.
\newblock American Mathematical Society, Providence, RI, 1988.
\newblock With a foreword by George E. Andrews.

\bibitem[Ono04]{Ono04}
Ken Ono.
\newblock The web of modularity: arithmetic of the coefficients of modular forms and $q$-series.
\newblock volume 102 of {\em CBMS Regional Conference Series in Mathematics}. Amer. Math. Soc., Providence, RI, 2004.

\bibitem[PR16]{Paule}
Peter Paule and Silviu Radu.
\newblock Partition analysis, modular functions, and computer algebra.
\newblock In {\em Recent trends in combinatorics}, volume 159 of {\em IMA Vol. Math. Appl.}, pages 511--543. Springer, [Cham], 2016.

\bibitem[Rad15]{Radu}
Cristian-Silviu Radu.
\newblock An algorithmic approach to {R}amanujan-{K}olberg identities.
\newblock {\em J. Symbolic Comput.}, 68:225--253, 2015.

\bibitem[Ram19]{Ramanujan}
S.~Ramanujan.
\newblock Some properties of $p(n)$, the number of partitions of $n$.
\newblock {\em Proc. Cambridge Philos, Soc.}, (19):210--213, 1919.

\bibitem[Smo21]{Smoot}
Nicolas~Allen Smoot.
\newblock On the computation of identities relating partition numbers in arithmetic progressions with eta quotients: an implementation of {R}adu's algorithm.
\newblock {\em J. Symbolic Comput.}, 104:276--311, 2021.

\end{thebibliography}

 \ 

{\textsc{Tewodros Amdeberhan}} 

\vspace{0.1in}

Department of Mathematics 

Tulane University

New Orleans, LA 70118, USA 

\vspace{0.05in}

{\tt tamdeber@tulane.edu}

 \vspace{0.3in}

{\textsc{Mircea Merca}} 

\vspace{0.1in}

Department of Mathematical Methods and Models

National University of Science and Technology Politehnica Bucharest

RO-060042 Bucharest, Romania

\vspace{0.05in}

{\tt mircea.merca@upb.ro}

\end{document}